\RequirePackage{fix-cm}
\documentclass[smallextended]{svjour3}     % onecolumn (second format)
\smartqed  % flush right qed marks, e.g. at end of proof
\usepackage{graphicx}
\usepackage{mathptmx}      % use Times fonts if available on your TeX system
\usepackage{latexsym}
\usepackage{amssymb,xspace}
\usepackage{amstext}
\usepackage{amsbsy,amsfonts}
\newcommand{\veps}{\varepsilon}

\def\RR{\mathbb R}
\def\NN{\mathbb N}

\def\dimh{\dim_{\mathcal H}}

\newcommand{\north}{\mathbf N}
 \newcommand{\bd}{B_{d+1}}
\newcommand{\sd}{\mathcal S_{d}}
\newcommand{\rd}{\mathbb R^{d+1}}

\begin{document}
\title{Boundary multifractal behaviour for harmonic functions in the ball
}


\author{Fr\'ed\'eric Bayart       \and
        Yanick Heurteaux %etc.
}


\institute{F. Bayart \at
              Clermont Universit\'e, Universit\'e Blaise Pascal, Laboratoire de Math\'ematiques, BP 10448, F-63000 CLERMONT-FERRAND -
CNRS, UMR 6620, Laboratoire de Math\'ematiques, F-63177 AUBIERE \\
              \email{Frederic.Bayart@math.univ-bpclermont.fr}           %  \\
           \and
           Y. Heurteaux \at
              Clermont Universit\'e, Universit\'e Blaise Pascal, Laboratoire de Math\'ematiques, BP 10448, F-63000 CLERMONT-FERRAND -
CNRS, UMR 6620, Laboratoire de Math\'ematiques, F-63177 AUBIERE \\
              \email{Yanick.Heurteaux@math.univ-bpclermont.fr}\\
              Tel.: +33 4 73 40 50 64\\
              Fax: +33 4 73 40 79 72\\
}

\date{Received: date / Accepted: date}

\maketitle
\begin{abstract}
It is well known that if $h$ is a nonnegative harmonic function in the ball of
$\mathbb R^{d+1}$ or if $h$ is harmonic in the ball with integrable
boundary values, then the radial limit of $h$ exists at almost every point of the
boundary. In this paper, we are interested in the exceptional set of points
of divergence and in the speed of divergence at these points. In particular, we prove that
for generic harmonic functions and for any $\beta\in [0,d]$, the Hausdorff
dimension of the set of points
$\xi$ on the sphere such that 
$h(r\xi)$ looks like $(1-r)^{-\beta}$ is equal to $d-\beta$. 
\end{abstract}

\maketitle

\section{Introduction}
The story of this paper begins in 1906, when P. Fatou proved in \cite{Fat03} 
that bounded
harmonic functions in the unit disk have nontangential limits almost 
everywhere on the circle.
Later on, this result was improved by Hardy and Littlewood in dimension 2, 
and by Wiener, Bochner and many others in arbitrary dimension (a complete
historical account can be found in \cite{Ste}). Let us also mention R. Hunt
and R. Wheeden who proved that a similar result holds for nonnegative
harmonic functions in Lipschitz domains (\cite{HW1,HW2}). 
To state the result of the 
nontangential convergence of harmonic functions in the ball, 
we need to introduce some terminology.

Let $d\geq 1$ and let $\sd$ (resp. $B_{d+1}$) be the (euclidean) unit sphere 
(resp. the unit ball) in $\mathbb R^{d+1}$. The euclidean norm in $\rd$ will 
be denoted by $\|\cdot\|$.  For $\mu\in\mathcal M(\sd)$, the set of complex Borel 
measures on $\sd$, the Poisson integral of $\mu$, denoted by $P[\mu]$, is 
the function on $B_{d+1}$ defined
by 
$$P[\mu](x)=\int_{\sd}P(x,\xi)d\mu(\xi),$$
where $P(x,\xi)$ is the Poisson kernel,
$$P(x,\xi)=\frac{1-\|x\|^2}{\|x-\xi\|^{d+1}}.$$
When $f$ is a function in $L^1(\sd)$, we denote simply by $P[f]$ the function 
$P[fd\sigma]$.
Here and elsewhere, $d\sigma$ denotes the normalized Lebesgue measure 
on $\sd$. 
For any $\mu\in\mathcal M(\sd)$, $P[\mu]$ is a harmonic function in $B_{d+1}$ 
and it is well known that, for instance, every bounded harmonic function in 
$B_{d+1}$ is the Poisson integral
$P[f]$ of a certain $f\in L^{\infty}(\sd)$. It is also well known that every 
nonnegative  harmonic function in $\bd$ is the Poisson integral $P[\mu]$ of a 
positive finite measure $\mu\in\mathcal M(\sd)$.

The Fatou theorem for Poisson integrals of $L^1$-functions says that, 
given a function $f\in L^1(\sd)$, 
then $P[f](ry)$ tends to $f(y)$  for almost every $y\in\sd$ when $r$ 
increases to 1. More generally, if $\mu\in\mathcal M(\sd)$, $P[\mu](ry)$ 
tends to $\frac{d\mu}{d\sigma}(y)$ almost everywhere and in fact, the 
limit exists for nontangential access.

In this paper, we are interested in the radial behaviour on exceptional sets, 
and especially in the
following questions. How quickly can  $P[f](ry)$ grow? For a prescribed growth 
$\tau(r)$, 
how big can be the sets of $y\in\sd$ such that 
$\limsup_{r\to 1}|P[f](ry)|/\tau(r)=+\infty$?
It is easy to see that the growth cannot be too fast. Indeed, the Poisson 
kernel satisfies, for any $y,\xi\in\sd$,
$$P(ry,\xi)\leq\frac2{(1-r)^d},$$
so that for any $f\in L^1(\sd)$, for any $y\in\sd$ and any $r\in(0,1)$,
$$P[f](ry)\leq \frac{2\|f\|_1}{(1-r)^d}.$$
This motivates us to introduce, for a fixed $\beta\in(0,d)$ and any 
$f\in L^1(\sd)$, the exceptional set
$$\mathcal E(\beta,f)=\left\{y\in\sd;\ \limsup_{r\to 1}\frac{|P[f](ry)|}
{(1-r)^{-\beta}}=+\infty\right\},$$
and we ask for the size of $\mathcal E(\beta,f)$. To measure the size of 
subsets of $\sd$, we
shall use the notion of Hausdorff dimension (see Section \ref{SECPREL} 
for precise definitions).
Our first main result is the following.
\begin{theorem}\label{THMMAIN1}
Let $\beta\in[0,d]$ and let $f\in L^1(\sd)$. Then 
$\dim_\mathcal H\big(\mathcal E(\beta,f)\big)\leq d-\beta$. Conversely, 
given a subset $E$ of $\sd$ such that $\dimh (E)<d-\beta$, 
there exists $f\in L^1(\sd)$ such that $E\subset \mathcal E(\beta,f)$.
\end{theorem}
The first part of Theorem \ref{THMMAIN1} has already been obtained by D. Armitage in \cite{Ar81} in the context of Poisson integrals on the upper half-space
(see also \cite{Wat80} for analogous results regarding solutions of the heat equation).
However, we will produce a complete proof of Theorem \ref{THMMAIN1}. Our method of proof differs substantially from that of \cite{Ar81}.
Moreover, it provides a more general result (see Theorem \ref{THMAUBRYLIKE} below). It seems that this last statement cannot be obtained
from Armitage's work without adding assumptions on $\phi$ and $\tau$.

Our second task is to perform a multifractal analysis of the radial behaviour of harmonic functions,
as is done in \cite{BH11}, \cite{BH11b} for the divergence of Fourier series. For a given function
$f\in L^1(\sd)$ and a given $y\in\sd$, we define the real number $\beta(y)$ as the infimum of the real numbers $\beta$ such that $|P[f](ry)|=O\left( (1-r)^{-\beta}\right)$. The level sets of the function $\beta$ are defined by
\begin{eqnarray*}
E(\beta,f)&=&\left\{y\in\sd;\ \beta(y)=\beta\right\}\\
&=&\left\{y\in\sd;\ \limsup_{r\to1}\frac{\log |P[f](ry)|}{-\log(1-r)}=\beta\right\}.
\end{eqnarray*}
We  can ask for which values of $\beta$ the sets
$E(\beta,f)$ are non-empty. This set of values will be called the domain of definition of the spectrum of singularities of $f$.
If $\beta$ belongs to the domain of definition of the spectrum of singularities, it is also interesting to estimate the
Hausdorff dimension of the sets $E(\beta,f)$. The function $\beta\mapsto
\dim_\mathcal{H}(E(\beta,f))$ will be called the spectrum of singularities of the
function $f$.

Theorem \ref{THMMAIN1} ensures that $\dim_\mathcal{H}(E(\beta,f))\leq d-\beta$ and our second main result
is that a \emph{typical} function $f\in L^1(\sd)$ satisfies $\dim_\mathcal{H}(E(\beta,f))=d-\beta $ for \emph{any} $\beta\in[0,d]$.
In particular, such a function $f$ has a multifractal behavior, in the sense that the domain of definition of its spectrum of singularities contains an interval with non-empty interior.
\begin{theorem}\label{THMMAIN2}
For quasi-all functions $f\in L^1(\sd)$, for any $\beta\in[0,d]$, $$\dim_\mathcal{H}\big(E(\beta,f)\big)=d-\beta.$$
\end{theorem}
The terminology "quasi-all" used here is relative to the Baire category theorem. It means that this property is true for a residual
set of functions in $L^1(\sd)$. 

\medskip

\noindent \textsc{Notations.} Throughout the paper, $\north=(0,\dots,0,1)$ will denote the north pole of $\sd$.
The letter $C$ will denote a positive constant whose value may change from line to line. This value may depend on the dimension $d$, but it will never depend on the other parameters which are involved.

\medskip

\noindent \textsc{Acknowledgements.} We thank the referee for his/her careful reading and for having provided to us references \cite{Ar81} and \cite{Wat80}.

\section{Preliminaries}\label{SECPREL}
In this section, we survey some results regarding Hausdorff measures.  We refer to \cite{Falc} and to \cite{Mat95} for more on this subject. Let $(X,d)$ be a metric space such that,
for every $\rho>0$, the space $X$ can be covered by a countable number of balls 
with diameter less than $\rho$. If $B=B(x,r)$  is a ball in $X$ and $\lambda>0$,
$|B|$ denotes the diameter of $B$ and $\lambda B$ denotes the ball $B$
scaled by a factor $\lambda$, i.e. $\lambda B=B(x,\lambda r)$.

A dimension function $\phi:\mathbb R_+\to\mathbb R_+$ is a continuous nondecreasing function satisfying $\phi(0)=0$. Given $E\subset X$, the $\phi$-Hausdorff outer measure of  $E$ is defined by 
$$\mathcal H^{\phi}(E)=\lim_{\veps\to 0}\inf_{r\in R_\veps(E)}\sum_{B\in r}\phi(|B|),$$
where $R_\veps(E)$ is the set of countable coverings of $E$ with balls $B$ with diameter $|B|\leq\veps$. 
When $\phi_s(x)=x^s$, we write for short $\mathcal H^s$ instead of $\mathcal H^{\phi_s}$. The Hausdorff dimension of a set $E$
is 
$$\dim_{\mathcal H}(E):=\sup\{s>0;\ \mathcal H^s (E)>0\}=\inf\{s>0;\ \mathcal H^s(E)=0\}.$$

We will need to construct on $\sd$ a family of subsets with prescribed Hausdorff dimension. 
For this we shall use results of \cite{BV06}. Recall that a function
$\phi:\mathbb R_+\to\mathbb R_+$ is doubling provided there exists $\lambda>1$ such that,
for any $x>0$, $\phi(2x)\leq\lambda \phi(x)$. From now on, we suppose that the metric space $(X,d)$ supports 
a doubling dimension function $\phi$ such that 
$$\frac1C \phi(|B|)\leq \mathcal H^\phi(B)\leq C\phi(|B|)$$
where $C$ is a positive constant independent of $B$. 

The previous assumption is satisfied when $X=\sd$, endowed with the distance inherited
from $\mathbb R^{d+1}$, and $\phi(x)=x^{d}$. 

Given a dimension function $\psi$ and a ball $B=B(x,r)$, we denote by $B^\psi$ the ball
$B^\psi=B(x,\psi^{-1}\circ\phi(r))$. The following mass transference principle of \cite{BV06} will be used.

\begin{lemma}[The mass transference principle]\label{LEMMTP}
Let $(B_i)$ be a sequence of balls in $X$ whose radii go to zero. Let $\psi$ be a dimension
function such that $\psi(x)/\phi(x)$ is monotonic and suppose that, for any ball $B$ in $X$,
$$\mathcal H^\phi\left(B\cap\limsup_{i\to+\infty}B_i\right)=\mathcal H^\phi(B).$$
Then, for any ball $B$ in $X$,
$$\mathcal H^\psi\left(B\cap\limsup_{i\to+\infty}B_i^\psi\right)=\mathcal H^\psi(B).$$
\end{lemma}

Finally, the following basic covering lemma due to Vitali will be required (see \cite{Mat95}).

\begin{lemma}[The $5r$-covering lemma]
Every family $\mathcal F$ of balls with uniformly bounded diameters in a
separable metric space $(X,d)$
contains a disjoint subfamily $\mathcal G$ such that
$$\bigcup_{B\in\mathcal F}B\subset \bigcup_{B\in\mathcal G}5B.$$
\end{lemma}

\section{Majorisation of the Hausdorff dimension}
Let $f\in L^1(\sd)$. We intend to show that $P[f](r\cdot)$ cannot grow too fast on sets with large Hausdorff dimension. More generally, we shall do this for $\mu\in\mathcal M(\sd)$ and $P[\mu]$ instead of $P[f]$. If $y \in\sd$ and $\delta>0$, we introduce
$$\kappa(y,\delta)=\big\{\xi\in\sd;\ \|\xi-y\|<\delta\big\}$$
the open spherical cap on $\sd$ with center $y$ and radius $\delta>0$. The set $\kappa(y,\delta)$ is just the ball with center $y$ and radius $\delta$ in the metric space $(\sd,\|\cdot\|)$. Let us also define the slice
$$\mathcal S(y,\delta_1,\delta_2)=\big\{\xi\in\sd;\ \delta_1\leq \|\xi-y\|<\delta_2\big\}$$
where $0\leq\delta_1<\delta_2$.

The starting point of our argument is a result linking the radial behaviour of $P[\mu]$ to the Hardy-Littlewood maximal function. More precisely, it is well known that if $y\in\sd$, then
$$\sup_{r\in(0,1)}\big|P[\mu](ry)\big|\leq \sup_{\delta>0}\frac{|\mu|(\kappa(y,\delta))}{\sigma(\kappa(y,\delta))}$$
(see for example \cite{ABR}). Our aim is to control, for a fixed $r$ close to 1, the minimal size of the caps which come into play on the
right-hand side.
\begin{lemma}\label{LEMHL}
Let $\mu\in \mathcal M(\sd)$, $r\in(0,1)$ and $y\in \sd$. There exists $\delta\geq 1-r$ such that
$$\big|P[\mu](ry)\big|\leq C\frac{|\mu|(\kappa(y,\delta))}{\sigma(\kappa(y,\delta))},$$
where $C$ is a constant independent of $\mu$, $r$ and $y$.
\end{lemma}
\begin{proof}
Replacing $\mu$ by $|\mu|$, we may assume that $\mu$ is positive. Moreover, without loss of generality, we may assume that  $y=\north$  is the north pole. Observe that
$$P[\mu](r\north)=\int_{\sd}P(r\north,\xi)d\mu(\xi),$$
with 
\begin{eqnarray*}
P(r\north,\xi)&=&\frac{1-r^2}{\|r\north-\xi\|^{d+1}}\\
&=&\frac{1-r^2}{(1-2r\xi_{d+1}+r^2)^{(d+1)/2}}.
\end{eqnarray*}
Observe also that $\|\xi-\north\|^2=2(1-\xi_{d+1})$ if $\xi\in \sd$. In particular, $P(r\north, \xi)$ just depends on $\|\xi-\north\|$ and $r$. Moreover, $P(r\north,\xi)$ decreases when $\|\xi-\north\|$ increases, $\xi$ keeping on $\sd$. 

We shall approximate $\xi\mapsto P(r\north,\xi)$ by functions which are constant on slices.
The function $\xi\mapsto P(r\north,\xi)$ is harmonic and nonnegative in the ball 
$$\{\xi\in\rd;\ \|\xi-\north\|<1-r\}.$$ By the Harnack inequality, there exists $C_0>0$ (which does not depend on $r$) such that, for any $\xi\in\rd$ with $\|\xi-\north\|\leq (1-r)/2$, 
$$P(r\north,\xi)\geq C_0P(r\north,\north).$$
Necessarily, $C_0$ belongs to $(0,1)$. We then define an integer  $k>0$ and  a finite sequence $\delta_0,\dots,\delta_k$ by
\begin{itemize}
\item $\delta_0=0$;
\item $\delta_1=(1-r)/2$;
\item $\delta_{j+1}$ (if it exists)  is the real number in $[\delta_j,2]$ such that $P(r\north,\xi^{j+1})=C_0P(r\north,\xi^j)$ where $\xi^j$ (resp. $\xi^{j+1}$) is an arbitrary point of
$\sd$ such that $\|\xi^j-\north\|=\delta_j$ (resp. $\|\xi^{j+1}-\north\|=\delta_{j+1}$) (remember that $P(r\north,\xi)$ only depends on $\|\xi-\north\|$);
\item $\delta_{j+1}=2$ and $k=j+1$  otherwise.  
\end{itemize}
Observe that the sequence is well defined  and that, by compactness, the process ends up after a finite number of steps.
We set $c_j=P(r\north,\xi^j)$, $0\leq j \leq k-1$ where $\xi^j$ is an arbitrary point in $\sd$ such that $\|\north-\xi^j\|=\delta_j$. Let us also remark  that, if $\xi\in\sd$,  $\xi\neq-\north$,
$$C_0\sum_{j=0}^{k-1}c_j \mathbf 1_{\mathcal S(\north,\delta_j,\delta_{j+1})}(\xi)\leq P(r\north,\xi)\leq \sum_{j=0}^{k-1}c_j \mathbf 1_{\mathcal S(\north, \delta_j,\delta_{j+1})}(\xi).$$
The sequence $(c_j)_{j\ge 0}$ is decreasing. Thus, we can rewrite the step function using only caps as
$$\sum_{j=0}^{k-1}c_j  \mathbf 1_{\mathcal S(\north,\delta_j,\delta_{j+1})}=
\sum_{j=1}^{k}d_j \mathbf 1_{\kappa(\north,\delta_j)}$$
where the real numbers $d_j$ are \emph{positive}. In fact, $d_1=c_0$ and $d_j=c_{j-1}-c_j$ if $j\ge 2$. 
Then we get
\begin{eqnarray}\label{EQLEMHL}
C_0\sum_{j=1}^{k} d_j\mathbf 1_{\kappa(\north,\delta_j)}\leq P(r\north,\xi)\leq \sum_{j=1}^{k}d_j \mathbf 1_{\kappa(\north,\delta_j)}.
\end{eqnarray} 
We integrate the right-hand inequality with respect to $\mu$ to obtain
\begin{eqnarray*}
P[\mu](r\north)&\leq&\sum_{j=1}^k d_j \mu(\kappa(\north,\delta_j))\\
&\leq&\sup_{j=1,\dots, k} \frac{\mu(\kappa(\north,\delta_j))}{\sigma(\kappa(\north,\delta_j))}\sum_{j=1}^k d_j \sigma(\kappa(\north,\delta_j))\\
&\leq&C_0^{-1}\sup_{j=1,\dots, k} \frac{\mu(\kappa(\north,\delta_j))}{\sigma(\kappa(\north,\delta_j))}\int_{\sd}P(r\north,\xi)d\sigma(\xi)
\end{eqnarray*}
where the last inequality is obtained by integrating the left part of (\ref{EQLEMHL}) over $\sd$ with respect to the surface measure $\sigma$. This yields the lemma, since $\int_{\sd}P(r\north,\xi)d\sigma(\xi)= 1$, except that
we have found a cap with radius greater than $(1-r)/2$ instead of $1-r$. Fortunately, it is easy
to dispense with the factor $1/2$. Indeed,
$$\frac{\mu(\kappa(\north,\delta))}{\sigma(\kappa(\north,\delta)\big)}\leq C\frac{\mu(\kappa(\north,\delta))}{\sigma(\kappa(\north,2\delta))}\leq C\frac{\mu(\kappa(\north,2\delta))}{\sigma(\kappa(\north,2\delta))}.$$
\end{proof}
The previous lemma is the main step to obtain an upper bound of the Hausdorff dimension of the sets where $P[\mu](r\cdot)$ behaves badly.
\begin{theorem}\label{THMAUBRYLIKE}
Let $\mu\in\mathcal M(\sd)$ and let $\tau:(0,1)\to(0,+\infty)$ be nonincreasing, with $\lim_{x\to 0^+}\tau(x)=+\infty$.
Let us define
$$\mathcal E(\tau,\mu)=\left\{y\in\sd;\ \limsup_{r\to 1}\frac{|P[\mu](ry)|}{\tau(1-r)}=+\infty\right\}.$$
Let $\phi:(0,+\infty)\to(0,+\infty)$ be a dimension function satisfying $\phi(s)=O(\tau(s)s^d)$. 
Then $$\mathcal H^\phi\big(\mathcal E(\tau,\mu)\big)=0.$$
\end{theorem}
\begin{proof}
For any $M>1$, we introduce
$$\mathcal E_M=\left\{y\in\sd;\  \limsup_{r\to 1}\frac{|P[\mu](ry)|}{\tau(1-r)}> M\right\}.$$
Let $\veps>0$ and $y\in \mathcal E_M$. The definition of $\mathcal E_M$ and Lemma \ref{LEMHL} ensure that  we can find  $r_y\in(0,1)$, as close to 1 as we want, and a cap $\kappa_y=\kappa(y,\delta_y)$  such that $\delta_y\geq 1-r_y$ satisfying 
\begin{eqnarray}\label{EQTAU}
M\tau(1-r_y)\leq |P[\mu](r_yy)|\leq C\frac{|\mu|(\kappa_y)}{\sigma(\kappa_y)}.
\end{eqnarray}
Observe that
$$\sigma(\kappa_y)\le\frac{C|\mu|(\sd)}{M\tau(1-r_y)}.$$
It follows that $\delta_y\to 0$ when $r_y\to 1$. We can then always ensure that $|\kappa_y|\le\veps$.
  The family $(\kappa_{y})_{y\in \mathcal E_M}$ is an $\veps$-covering of $\mathcal E_M$.
By the 5r-covering lemma, one can extract from it a countable family of disjoint caps $(\kappa_{y_i})_{i\in\NN}$ such that $\mathcal E_M\subset\bigcup_i 5\kappa_{y_i}$. Inequality (\ref{EQTAU}) implies that
$$M\sum_i \tau(1-r_{y_i})\sigma(\kappa_{y_i})\leq C\|\mu\|.$$
If we remark that $|5\kappa_{y_i}|\ge\delta_{y_i}\ge1-r_{y_i}$, we can conclude that
$$\sum_i \tau(|5\kappa_i|)|5\kappa_i|^d\leq \frac{C}M \|\mu\|.$$
Our assumption on $\phi$ ensures that $\mathcal H^{\phi}(\mathcal E_M)\leq C(\phi,\mu)/M$. The result follows from the equality $\mathcal E(\tau,\mu)=\bigcap_{M>1}\mathcal E_M$.
\end{proof}
Applying this to the function $\tau(s)=s^{-\beta}$, we get the first half of Theorem \ref{THMMAIN1}.
\begin{corollary}\label{CORAUBRYLIKE}
For any $\beta\in[0,d]$, for any $\mu\in \mathcal M(\sd)$, $\dimh\big(\mathcal E(\beta,\mu)\big)\leq d-\beta$.
\end{corollary}
\begin{remark}The corresponding result for the divergence of Fourier series was obtained in
 \cite{Aub06} using the Carleson-Hunt theorem (see also \cite{BH11b} for the $L^1$-case). Our proof
 in this context is much more elementary, since we do not need the maximal inequality for the 
 Hardy-Littlewood maximal function.
\end{remark}

\section{Minorisation of the Hausdorff dimension}
In this section, we prove the converse part of Theorem \ref{THMMAIN1}. We first need a technical lemma on the Poisson kernel. 
\begin{lemma}\label{LEMPOISSONKERNEL}
There exists a constant $C>0$ such that, for any $r\in(1/2,1)$ and any $y\in\sd$,
$$\int_{\kappa(y,1-r)}P(ry,\xi)d\sigma(\xi)\geq C.$$
\end{lemma}
\begin{proof} We may assume $y=\north$. Let $\rho=1-r$. A generic point $x=(x_1,\cdots,x_{d+1})\in \RR^{d+1}$ will be denoted by $x=(x',x_{d+1})$ with $x'\in\mathbb R^d$. In particular, $x\in\kappa(\north,\rho)$ if and only if $\|x'\|^2+x_{d+1}^2=1$ and  
$\|x'\|^2+(1-x_{d+1})^2<\rho^2$. 
%It follows that $x_{d+1}\ge 1-\rho^2/2$
%and $\|x'\|^2\le \rho^2-\rho^4/4$, the maximal value being obtained when $x_{d+1}= 1-\rho^2/2$.
Let $\mathcal C$ be the cylinder
$$\mathcal C=\left\{ x\in\mathbb R^{d+1}\ ;\ \|x'\|^2<\rho^2/2\ \mbox{and}\ 1-2\rho< x_{d+1}< 1\right\}.$$
It is not hard to show that $\mathcal S_d\cap \overline{\mathcal C}\subset\kappa(\mathbf N,\rho)$ when $1/2<r<1$. We
now define two harmonic functions: $h$ is the harmonic function in $\mathcal C$ such that 
$h(x)=1$ if $x\in\partial \mathcal C\cap\{ x_{d+1}=1\}$ and $h(x)=0$ if $x\in\partial \mathcal C\cap\{ x_{d+1}<1\}$; 
$u$ is the harmonic function in $B_{d+1}$ such that $u=1$ on $\kappa(\mathbf N,\rho)$ and
$u=0$ elsewhere on $\mathcal S_d$ ($h$ and $u$ are the Perron-Wiener-Brelot solutions of the Dirichlet problem with the given boundary data).
We claim that $h\leq u$ on $\partial(\mathcal C\cap B_{d+1})$. Indeed, we can decompose 
$\partial (\mathcal C\cap B_{d+1})$ into $E\cup F$, with $E\subset S_d\cap\overline{\mathcal C}$
and $F\subset \partial \mathcal C\cap\{x_{d+1}<1\}$. Now, $u=1\geq h$ on $E$ and $u\ge 0=h$ on $F$. By the maximum principle in $\mathcal C\cap B_{d+1}$, we deduce that $u(x)\geq h(x)$ for any $x\in \mathcal C\cap B_{d+1}$. In particular this holds for $x=(1-\rho)\north=r\north$, so that 
$$\int_{\kappa(\north,\rho)}P(r\mathbf N,\xi)d\sigma(\xi)\ge h(r\mathbf N).$$
On the other hand, $\mathcal C$ is just the translation and dilation of a fixed domain : $\mathcal C=\mathbf N+\rho\mathcal U$, where
$$\mathcal U=\left\{ x\in\mathbb R^{d+1}\ ;\ \|x'\|^2<1/2\ \mbox{and}\ -2< x_{d+1}< 0\right\}.$$
Thus the quantity $h(r\north)$ is strictly positive and independent of $r$. We can then take $C=h(r\north)$.
\begin{center}
\includegraphics[width=13cm]{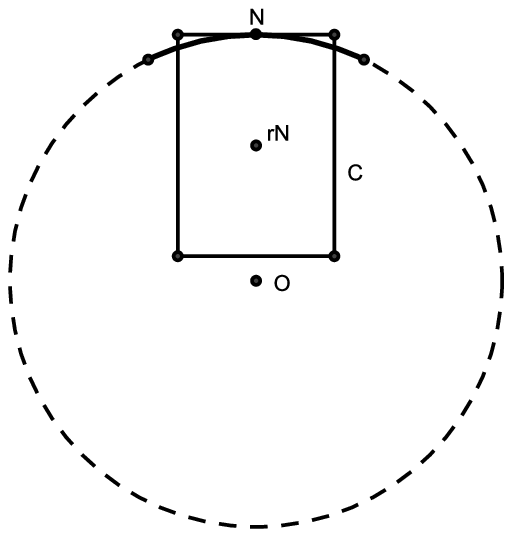}
\end{center}
\end{proof}
Here is the converse part of Theorem \ref{THMMAIN1}.
\begin{theorem}\label{THMAUBRYLIKE2}
Let $E\subset\sd$, let $\phi$ be a dimension function and let $\tau:(0,1)\to(0,+\infty)$
be nonincreasing with $\lim_{x\to 0^+}\tau(x)=+\infty$. Suppose that $\mathcal H^{\phi}(E)=0$ and that $\tau(s)=O\big(s^{-d}\phi(s)\big)$. Then there exists $f\in L^1(\sd)$ such that, for any $y \in E$, 
$$\limsup_{r\to 1}\frac{P[f](ry)}{\tau(1-r)}=+\infty.$$
\end{theorem}
A remarkable feature of Theorem \ref{THMAUBRYLIKE} and Theorem \ref{THMAUBRYLIKE2} is that they are sharp: if $\phi(s)=\tau(s)s^d$ is a dimension function and
$$\mathcal E(\tau,f)=\left\{y\in\sd;\ \limsup_{r\to 1}\frac{|P[f](ry)|}{\tau(1-r)}=+\infty\right\},$$
then
\begin{enumerate}
\item for any $f\in L^1(\sd)$, $\mathcal H^\phi\big(\mathcal E(\tau,f)\big)=0$;
\item if $E$ is a set satisfying $\mathcal H^\phi(E)=0$, we can find $f\in L^1(\sd)$ such that
$\mathcal E(\tau,f)\supset E$. 
\end{enumerate}
\begin{proof}[Proof of Theorem \ref{THMAUBRYLIKE2}]
Let $j\ge 1$. Since $\mathcal H^\phi(E)=0$, we can find a covering $\mathcal R_j$ of $E$
by caps with diameter less than $2^{-j}$ and such that $\sum_{\kappa\in\mathcal R_j} \phi(|\kappa|)\leq 2^{-j}.$ We collect together the caps with approximately the same size. Precisely, if $n\geq 1$, let
$$\mathcal C_n=\left\{\kappa\in\bigcup_j \mathcal R_j;\ 2^{-(n+1)}<|\kappa|\leq 2^{-n}\right\}.$$
Let also $E_n=\bigcup_{\kappa\in\mathcal C_n}\kappa$ so that $E\subset\limsup_n E_n$ and
$$\sum_{n\ge 1}\sum_{\kappa\in\mathcal C_n}\phi(|\kappa|)\leq\sum_{j\ge 1} \sum_{\kappa\in\mathcal R_j}
\phi(|\kappa|)\leq 1.$$
In particular, there exists a sequence $(\omega_n)_{n\ge 1}$ tending to infinity such that
$$\sum_{n\ge 1} \sum_{\kappa\in\mathcal C_n}\omega_n \phi(|\kappa|)<+\infty.$$
For any $n\geq 1$, let $x_{n,1},\dots,x_{n,{m_n}}$ be the centers of the caps
appearing in $\mathcal C_n$ and let
$\kappa_{n,i}=\kappa(x_{n_,i},2\cdot2^{-n})$. 
We define
$$f=\sum_{n\ge 1} \sum_{i=1}^{m_n}\omega_n \tau(2^{-n})\mathbf 1_{\kappa_{n,i}}.$$
$f$ belongs to $L^1(\sd)$. Indeed, 
\begin{eqnarray*}
\|f\|_1&\leq&C\sum_{n\ge 1}\sum_{i=1}^{m_n}\omega_n \tau(2^{-n})(2^{-n})^d\\
&\leq&C\sum_{n\ge 1}\sum_{i=1}^{m_n}\omega_n\phi(2^{-n})\\
&\leq&C\sum_{n\ge 1} \sum_{\kappa\in\mathcal C_n}\omega_n \phi(|\kappa|)<+\infty.
\end{eqnarray*}
Moreover, let $y\in E_n$ and let $r=1-2^{-n}$. Let also
$\kappa_y=\kappa(x_{n,i},\delta_{n,i})\in\mathcal C_n$ 
such that $y$ belongs to $\kappa_y$. It is clear that $\|y-x_{n,i}\|\le\delta_{n,i}\le 2^ {-n}$ so that 
$\kappa(y,2^{-n})\subset \kappa_{n,i}$. By the positivity of $f$ and of the Poisson kernel,
\begin{eqnarray*}
P[f](ry)&\geq&\int_{\kappa(y,2^{-n})}\omega_n\tau(2^{-n})P(ry,\xi)d\sigma(\xi)\\
&\geq&C \omega_n\tau(1-r)
\end{eqnarray*}
where $C$ is the constant that appears in  Lemma \ref{LEMPOISSONKERNEL}. 
Thus, provided $y$ belongs to $\limsup_n E_n$, we get 
$$\limsup_{r\to 1}\frac{P[f](ry)}{\tau(1-r)}=+\infty,$$
which is exactly what we need.
\end{proof}

\section{Construction of saturating functions}
In this section, we turn to the construction of functions in $L^1(\sd)$ having multifractal behaviour. Our first step is a construction of a sequence of nets in $\sd$ which play the same role as dyadic numbers in the interval.
\begin{lemma}
There exists a sequence  $(\mathcal R_n)_{n\geq 1}$ of finite subsets of $\mathcal S^d$  satisfying 
\begin{itemize}
\item $\mathcal R_n\subset \mathcal R_{n+1}$;
\item $\bigcup_{x\in\mathcal R_n}\kappa(x,2^{-n})=\sd$;
\item $\mathop{\rm card}\,(\mathcal R_n)\leq C2^{nd}$;
\item For any $x,y$ in $\mathcal R_n$, $x\neq y$, then $|x-y|\geq 2^{-n}$.
\end{itemize}
\end{lemma}
\begin{proof}
Let $\mathcal R_0=\varnothing$ and let us explain how to construct $\mathcal R_{n+1}$
from $\mathcal R_n$. $\mathcal R_{n+1}$ is a maximal subset of $\sd$ containing $\mathcal R_n$
and such that any distinct points in $\mathcal R_{n+1}$ have their distance greater than or
equal to $2^{-(n+1)}$. Then $\bigcup_{x\in\mathcal R_{n+1}}\kappa\left(x,2^{-(n+1)}\right)=\sd$ 
by maximality of $\mathcal R_{n+1}$. Then, taking the surface and using that the caps $\kappa\left(x,2^{-(n+2)}\right)$, 
$x\in\mathcal R_{n+1}$, are pairwise disjoint, we get
$$\mathop{\rm card}\,(\mathcal R_{n+1})\times C2^{-(n+2)d}\leq 1.$$
\end{proof}
From now on, we fix a sequence $(\mathcal R_n)_{n\ge 0}$ as in the previous lemma. Our sets with big Hausdorff dimension will be based on open caps centered at points of $\mathcal R_n$. Precisely, let $\alpha>1$ and let $N_{n,\alpha}=[n/\alpha]+1$ where $[n/\alpha]$ denotes the integer part of $n/\alpha$. We introduce
$$D_{n,\alpha}=\bigcup_{x\in \mathcal R_{N_{n,\alpha}}}\kappa\left(x,2^{-n}\right).$$

\begin{lemma}\label{LEMHD}
Let $\alpha>1$ and let $(n_k)_{k\ge 0}$ be a sequence of integers growing to infinity. Then
$$\mathcal H^{d/\alpha}\left(\limsup_{k\to +\infty}D_{n_k,\alpha}\right)=+\infty.$$
\end{lemma}
\begin{proof}
This follows from an application of the mass transference principle (Lemma \ref{LEMMTP}), applied with the function
$\psi(x)=x^{d/\alpha}$ and $\phi(x)=x^d$. The key points are that
$$\bigcup_{x\in\mathcal R_{N_{n,\alpha}}}\kappa\left(x,2^{-N_{n,\alpha}}\right)=\sd$$
and that $\kappa\left(x,2^{-n}\right)\supset \kappa\left(x,\psi^{-1}\circ \phi(2^{-N_{n,\alpha}})\right)$ since $\alpha N_{n,\alpha}\geq n$.
\end{proof}

We now construct saturating functions step by step.

\begin{lemma}\label{LEMSAT}
Let $n\geq 1$. There exists a nonnegative fonction  $f_n\in L^1(\sd)$, satisfying $\|f_n\|_1=1$, such that,  for any $\alpha>1$, 
for any $y\in D_{n,\alpha}$,
$$P[f_n](r_ny)\geq\frac C{n}2^{(n-N_{n,\alpha})d},$$
where $1-r_n={2^{-n}}$, $N_{n,\alpha}=[n/\alpha]+1$ and $C$ is independent of $n$ and $\alpha$.
\end{lemma}
%Here $\gamma$ denotes the constant appearing in Lemma \ref{LEMPOISSONKERNEL}.
\begin{proof}
We define $\tilde f_n$ by
$$\tilde f_n:=\frac 1{n+1}\sum_{N=1}^{n+1} \sum_{x\in\mathcal R_N}2^{(n-N)d}\mathbf 1_{\kappa(x,2\cdot 2^{-n})}.$$
The triangle inequality ensures that
\begin{eqnarray*}
\|\tilde f_n\|_1&\leq&\frac C{n+1}\sum_{N=1}^{n+1} \mathop{\rm card}\,(\mathcal R_N)2^{(n-N)d}2^{-nd}\\
&\leq &C.
\end{eqnarray*}
 Let $y\in D_{n,\alpha}$ and let $x\in\mathcal R_{N_{n,\alpha}}$ such that
 $y\in \kappa\left(x,2^{-n}\right)$. Observe that
 $\kappa\left(y,2^{-n}\right)\subset \kappa\left(x,2.2^{-n}\right)$. Moreover,
 $1\le N_{n,\alpha}\le n+1$. Using  the positivity  of the Poisson kernel, we get
$$P[\tilde f_n](ry)\geq\int_{\kappa\left(y,2^{-n}\right)}\frac{2^{(n-N_{n,\alpha})d}}{n+1}P(ry,\xi)d\sigma(\xi).$$
 Lemma \ref{LEMPOISSONKERNEL} ensures that
$$P[\tilde f_n](r_ny)\ge \frac{C}{n+1}2^{(n-N_{n,\alpha})d} $$
and it suffices to take $f_n=\frac{\tilde f_n}{\Vert \tilde f_n\Vert_1}$.
\end{proof}

We are now ready for the proof of our second main theorem.
\begin{proof}[Proof of Theorem \ref{THMMAIN2}]
Let $(g_n)_{n\ge 1}$ be a dense sequence of $L^1(\sd)$ such that each $g_n$ is continuous and $\|g_n\|_\infty\leq n$. The maximum principle ensures that  for any $r\in(0,1)$ and for any 
$\xi\in\sd$,
$$|P[g_n](r\xi)|\leq n.$$
Let $(f_n)$ be the sequence given by Lemma \ref{LEMSAT} and let us set
$$h_n=g_n+\frac 1n f_n.$$
$(h_n)_{n\ge 1}$ remains dense in $L^1(\sd)$. Moreover, if  $r_n=1-{2^{-n}}$,   $\alpha>1$ and  $y\in D_{n,\alpha}$, 
\begin{eqnarray*}
P[h_n](r_ny)&\geq&C\frac{2^{\left(n-N_{n,\alpha}\right)d}}{n^2}-n\\
&\geq& C \frac{2^{\left(n-N_{n,\alpha}\right )d}}{2n^2}
\end{eqnarray*}
provided $n$ is sufficiently large. Let us finally consider  $\delta_n>0$ sufficiently small  such that
$$\|P[f](r_n\cdot)\|_\infty\leq 1\quad\mbox{if}\quad \|f\|_1\le \delta_n.$$
The residual set we will consider is the dense $G_\delta$-set
$$A=\bigcap_{l\ge 1}\bigcup_{n\geq l}B_{L^1}(h_n,\delta_n).$$
Pick any $f\in A$. One can find an increasing sequence of integers $(n_k)$
such that $f\in B_{L^1}(h_{n_k},\delta_{n_k})$ for any $k$. Let $\alpha>1$
and let $y\in\limsup_k D_{n_k,\alpha}=:D_\alpha(f)$.
Then we can find integers $n$, picked in the sequence $(n_k)_{k\ge 1}$, as large as we want such that
$$P[f](r_n y)\geq P[h_n](r_ny)-1\geq   C \frac{2^{\left(n-N_{n,\alpha}\right)d}}{2n^2}-1.$$
Observe that for such values of $n$,
$$\frac{\log|P[f](r_ny)|}{-\log(1-r_n)}\ge\frac{\left( n-N_{n,\alpha}\right)d}{n}+o(1).$$
Hence, 
$$\limsup_{r\to 1}\frac{\log |P[f](ry)|}{-\log(1-r)}\geq \lim_{n\to +\infty}\left(1-\frac{N_{n,\alpha}}n\right)d=\left(1-\frac1\alpha\right)d.$$
Furthermore, Lemma \ref{LEMHD} tells us that $\mathcal H^{d/\alpha}(D_\alpha(f))=+\infty$.
We divide $D_\alpha(f)$ into two parts:
\begin{eqnarray*}
D_\alpha^{(1)}(f)&=&\left\{y\in D_\alpha(f);\ \limsup_{r\to 1}\frac{\log |P[f](ry)|}{-\log(1-r)}=\left(1-\frac1\alpha\right)d\right\}\\
D_\alpha^{(2)}(f)&= &\left\{y\in D_\alpha(f);\ \limsup_{r\to 1}\frac{\log |P[f](ry)|}{-\log(1-r)}> \left(1-\frac1\alpha\right)d\right\}.
\end{eqnarray*}
Let $(\beta_n)_{n\ge 0}$ be a sequence of real numbers such that $$\beta_n>\left(1-\frac1\alpha\right)d\quad\mbox{and}\quad\lim_{n\to +\infty} \beta_n=\left(1-\frac1\alpha\right)d.$$
Then 
$$D_\alpha^{(2)}(f)\subset\bigcup_{n\geq 0}\mathcal E(\beta_n,f).$$
Observe that $\frac{d}\alpha>d-\beta_n$. Then,  by Corollary \ref{CORAUBRYLIKE}, 
$\mathcal H^{d/\alpha}(\mathcal E(\beta_n,f))=0$. We get  
$$\mathcal  H^{d/\alpha}(D_\alpha^{(2)}(f))=0\quad\mbox{and}\quad\mathcal H^{d/\alpha}(D_\alpha^{(1)}(f))=+\infty.$$
Finally, 
$$E\left(\left(1-\frac1\alpha\right)d,f\right)\supset D_\alpha^{(1)}(f)$$
and
$$\dim_{\mathcal H}\left(E\left(\left(1-\frac1\alpha\right)d,f\right)\right)\geq \frac d\alpha.$$
By Corollary \ref{CORAUBRYLIKE} again, this inequality is necessarily an equality, and we conclude
that $f$ satisfies the conclusion of Theorem \ref{THMMAIN2} by setting 
$$\left(1-\frac1\alpha\right)d=\beta\iff \frac d\alpha=d-\beta.$$
\end{proof}

One can also ask whether the Poisson integral of a typical Borel measure on $\sd$ has a multifractal behaviour.
Here, we have to take care of the topology on $\mathcal M(\sd)$. We endow it with the
weak-star topology, which turns the unit ball $B_{\mathcal M(\sd)}$ of the dual space $\mathcal M(\sd)$ into a compact space. We need the following folklore lemma:
\begin{lemma}
The set of measures $fd\sigma$, with $f\in\mathcal C(\sd)$, is weak-star dense in $\mathcal M(\sd)$.
\end{lemma}
\begin{proof}
The set of measures with finite support is weak-star dense in $\mathcal M(\sd)$ (see for instance \cite{Bil}). Thus, let $\xi\in\sd$, let $\veps>0$ and let $g_1,\dots,g_n\in\mathcal C(\sd)$. It suffices to prove that one can find $f\in\mathcal C(\sd)$ such that, for any $\veps>0$, for
any $i\in\{1,\dots,n\}$, 
$$\left|g_i(\xi)-\int_{\sd}g_i(y)f(y)d\sigma(y)\right|<\veps.$$
Since each $g_i$ is continuous at $\xi$, one can find $\delta>0$ such that $|\xi-y|<\delta$ implies $|g_i(\xi)-g_i(y)|<\veps$. Let $f$ be a continuous and nonnegative function on $\sd$ with support
in $\kappa(\xi,\delta)$ and whose integral is equal to 1. Then
\begin{eqnarray*}
\left|g_i(\xi)-\int_{\sd}g_i(y)f(y)d\sigma(y)\right|&\leq&\int_{\kappa(\xi,\delta)}|g_i(\xi)-g_i(y)|f(y)d\sigma(y)\\
&\leq&\veps.
\end{eqnarray*}
\end{proof}
Mimicking the proof of Theorem \ref{THMMAIN2}, we can prove the following result.
\begin{theorem}
For quasi-all measures $\mu\in B_{\mathcal M(\sd)}$, for any $\beta\in [0,d]$, 
$$\dimh\big(E(\beta,\mu)\big)=d-\beta.$$
\end{theorem}
\begin{proof}
Let $(g_n)_{n\ge 1}$ be a dense sequence of the unit ball of $\mathcal C(\sd)$ such that $\|g_n\|_\infty\leq 1-\frac1n$. The sequence $(g_nd\sigma)_{n\ge 1}$ is weak-star dense in $B_{\mathcal M(\sd)}$. Let $(f_n)_{n\ge 1}$ be the sequence given by Lemma \ref{LEMSAT} and let us set
$$h_n=g_n+\frac1n f_n$$
so that $(h_nd\sigma)_{n\ge 1}$ lives in the unit ball $B_{\mathcal M(\sd)}$ and is always a weak-star dense sequence in $B_{\mathcal M(\sd)}$. For any $\alpha>1$ and any $y\in D_{n,\alpha}$, 
\begin{eqnarray*}
P[h_n](r_ny)&\geq& C \frac{2^{\left(n-N_{n,\alpha} \right)d}}{n^2}-1
\end{eqnarray*}
with $r_n=1- 2^{-n}$. The function $(y,\xi)\mapsto P(r_n y,\xi)$ is uniformly continuous on $\sd\times\sd$. In particular, using the compactness of $\sd$, one may find $y_1,\dots,y_s\in\sd$ such that, for any $y\in\sd$, 
there exists $j\in\{1,\dots,s\}$ satisfying
$$\forall \xi\in\sd,\quad \big|P(r_ny,\xi)-P(r_ny_j,\xi)\big|\leq 1.$$
Let $\mathcal U_n$ be the following weak-star open neighbourhood of $h_nd\sigma$ in
$B_{\mathcal M(\sd)}$:
\begin{eqnarray*}
\mathcal U_n&=&\Bigg\{\mu\in B_{\mathcal M(\sd)};\ \textrm{for all }j\in\{1,\dots,s\},\\
&&\quad\quad \left|\int_{\sd}P(r_n y_j,\xi)d\mu-\int_{\sd}P(r_n y_j,\xi)h_n(\xi)d\sigma\right|<1\Bigg\}.
\end{eqnarray*}
By the triangle inequality, for any $y\in\sd$ and any $\mu\in\mathcal U_n$,
$$|P[\mu-h_nd\sigma]|(r_ny)|\leq 3.$$
We now define $A=\bigcap_{l\ge 1}\bigcup_{n\geq l}\mathcal U_n$ which is a
dense 
$G_\delta$-subset
of $B_{\mathcal M(\sd)}$, and we conclude as in the proof of Theorem \ref{THMMAIN2}.
\end{proof}
If we remember that $\mu\mapsto P[\mu]$ is a bijection between the set of
nonnegative finite measures on the sphere $\sd$ and the set of 
nonnegative harmonic functions in the ball 
$B_{d+1}$ we can also obtain the following result.
\begin{theorem}\label{THMPOSITIVE}
For quasi-all nonnegative harmonic functions $h$ in the unit ball $B_{d+1}$, for any
$\beta\in [0,d]$,
$$\dimh\big(E(\beta,h)\big)=d-\beta$$
where $E(\beta,h)$ is defined here by $\displaystyle E(\beta,h)=\left\{ y\in\sd\ ;\ \limsup_{r\to 1}\frac{\log h(ry)}{-\log(1-r)}=\beta\right\}$. 
\end{theorem}
The set $\mathcal H^+(\bd)$ of nonnegative harmonic functions in the unit ball
$\bd$ is endowed with the topology of the locally uniform convergence. It
is a closed cone in the complete vector space of all continous functions in the
ball. So it satisfies the Baire's property.
\begin{proof}[Proof of Theorem \ref{THMPOSITIVE}.]We begin with the following
  lemma.
\begin{lemma}\label{LEMDENSE}
The set of nonnegative functions which are continuous in the closed unit ball
$\overline{\bd}$ and harmonic in the open ball $\bd$ is dense in
$\mathcal H^+$.
\end{lemma}
\begin{proof}
Let $h\in\mathcal H^+$ and $\rho_n<1$ be a sequence of real number that
increases to 1. Set $f_n(\xi)=h(\rho_n\xi)$ if $\xi\in\sd$ and 
$h_n(x)=h(\rho_nx)=P[f_n](x)$ if $x\in \bd$. The functions $h_n$ are
nonnegative, harmonic and continuous
on the closed ball $\overline{\bd}$. Moreover, let $\rho<1$. The uniform
continuity of $h$ in the closed ball $\bar B(0,\rho)=\{x\ ;\ \|x\|\le\rho\}$
ensures that $h_n$ converges uniformly to $h$ in the compact set $\bar
B(0,\rho)$.   
\end{proof}
We can now prove Theorem \ref{THMPOSITIVE}, using the same way as in Theorem
\ref{THMMAIN2}. Let $(g_n)_{n\ge 1}$ be a dense sequence in the set 
of nonnegative continuous functions in $\sd$. 
Lemma \ref{LEMDENSE} ensures that the sequence $(P[g_n])_{n\ge
  1}$ is dense in $\mathcal H^+$. Moreover, we can suppose that 
$\| g_n\|_\infty\le n$ so that by the 
maximum principle, $0\le P[g_n](x)\le n$ for any $x\in \bd$. Let
$(f_n)_{n\ge 1}$ be the sequence given by Lemma \ref{LEMSAT} and observe that
if $\|x\|\le\rho$,
$$\left|\frac1nP[f_n](x)\right|\le\frac2{n(1-\rho)^d}\|f_n\|_1=
\frac2{n(1-\rho)^d}.$$
It follows that $\frac1nP[f_n]$ goes to 0 in $\mathcal H^+$. Define
$$h_n=P[g_n]+\frac1n P[f_n]$$ 
so that $(h_n)_{n\ge 1}$ is always dense in $\mathcal H^+$. Let $\alpha>1$,  
$y\in D_{n,\alpha}$ and $r_n=1- 2^{-n}$. Lemma \ref{LEMSAT} ensures that
\begin{eqnarray*}
h_n(r_ny)&\geq& C \frac{2^{\left(n-N_{n,\alpha} \right)d}}{n^2}-n.
\end{eqnarray*}
We can define 
$$A=\bigcap_{l\ge 1}\bigcup_{n\ge l}\left\{ h\in \mathcal H^+\ ;\
  \sup_{\|x\|\le r_n}|h(x)-h_n(x)|<1\right\}$$
which is a dense $G_\delta$-set in $\mathcal H^+$ and we can
conclude as in the proof of Theorem \ref{THMMAIN2}.
\end{proof}

\end{document}